\documentclass[10pt,oneside,a4paper]{amsart}
\usepackage{latexsym}
\usepackage{amsfonts,amsmath,amssymb,indentfirst, tikz}
\usepackage[english]{babel}
\usepackage[all]{xy}
\usepackage[pdftex]{hyperref}
\usepackage{tikz-cd}

\newcommand{\bdism}{\begin{displaymath}}
\newcommand{\edism}{\end{displaymath}}
\newcommand{\cc}{\mathbb{C}}
\newcommand{\rr}{\mathbb{R}}

\newcommand{\zz}{\mathbb{Z}}

\newcommand{\pp}{\mathbb{P}}

\newcommand{\B}{\mathbb{B}}

\newcommand{\Del}{\Delta}
\newcommand{\Gam}{\Gamma}

\newcommand{\sig}{\sigma}
\newcommand{\om}{\omega}

\newcommand{\ol}{\overline}

\newcommand{\ssm}{\smallsetminus}

\DeclareMathOperator{\PU}{PU}

\newtheorem{theorem}{Theorem}[section]
\newtheorem{proposition}[theorem]{Proposition}
\newtheorem{corollary}[theorem]{Corollary}
\newtheorem{lemma}[theorem]{Lemma}

\address{Department of Mathematics, Stony Brook University, Stony Brook, NY 11794 - 3651} \email{lucafabrizio.dicerbo@stonybrook.edu}
\address{Department of Mathematics 1805 N. Broad Street, Philadelphia, PA 10122, USA} \email{mstover@temple.edu}

\author{Luca F. Di Cerbo, Matthew Stover}

\title{\bf Punctured spheres in complex hyperbolic surfaces and bielliptic ball quotient compactifications}

\begin{document}

\begin{abstract}
In this paper, we study punctured spheres in two dimensional ball quotient compactifications $(X, D)$. For example, we show that smooth toroidal compactifications of ball quotients cannot contain properly holomorphically embedded $3$-punctured spheres. We also use totally geodesic punctured spheres to prove ampleness of $K_X + \alpha D$ for $\alpha \in (\frac{1}{4}, 1)$, giving a sharp version of a theorem of the first author with G.\ Di Cerbo. Finally, we produce the first examples of bielliptic ball quotient compactifications modeled on the Gaussian integers.
\end{abstract}

\maketitle


\section{Introduction}
\pagestyle{myheadings} \markboth{Right}{\textbf{Punctured spheres in complex hyperbolic surfaces}}
\pagenumbering{arabic}

Throughout this paper, a \emph{ball quotient compactification} will mean the smooth toroidal compactification of a quotient $\B^2 / \Gam$ of the unit ball in $\cc^2$ with its Bergman metric. Here $\Gam \subset \PU(2,1)$ is a torsion-free nonuniform lattice such that all parabolic elements are rotation-free. See \cite{Ash} and \cite{Mok} for details.

The primary purpose of this paper is to study holomorphically immersed or embedded spheres in ball quotient compactifications. For instance, we will prove the following results.

\begin{theorem}\label{thm:No3s}
Let $X$ be a ball quotient compactification with compactification divisor $D$.
\begin{enumerate}
\item If $C_0 \subset X \ssm D$ is a holomorphically immersed totally geodesic submanifold, then $C_0$ has an even number of cusps.

\item There are no properly holomorphically embedded $3$-punctured spheres on $X \ssm D$ arising from smooth rational curves on $X$.

\item If $X \ssm D$ contains a totally geodesic 4-punctured $\pp^{1}$, then $X$ is not minimal and the totally geodesic 4-punctured $\pp^{1}$ determines one of the exceptional curves on $X$.

\item If $X \ssm D$ contains a totally geodesic 6-punctured $\pp^{1}$, then $K_X$ is not ample and the totally geodesic 6-punctured $\pp^{1}$ determines a smooth rational $(-2)$-curve on $X$.

\end{enumerate}
\end{theorem}

We prove these, along with a number of related results, in \S \ref{sec:Punctures}. The method of proof is by studying the restriction of the complex hyperbolic metric on $X \ssm D$ to a curve on $X$. One can prove part (1) using the Hirzebruch--H\"ofer relative proportionality principle \cite{HH}, but we give an elementary differential-geometric proof for the sake of completeness. One should compare Theorem \ref{thm:No3s}(2) with the well-understood case of $3$-punctured spheres on finite-volume hyperbolic $3$-manifolds. It is a theorem of Colin Adams \cite{Ada85} that any properly embedded essential $3$-punctured sphere in a hyperbolic $3$-manifold must be isotopic to one that is totally geodesic. One can then produce many examples of complete finite-volume hyperbolic $3$-manifolds containing totally geodesic $3$-punctured spheres via knots in $S^3$. It would also be interesting to know whether or not any properly embedded essential $3$-punctured sphere in a ball quotient must be real totally geodesic. We suspect one could determine this via the Toledo invariant.

\medskip

One consequence of Theorem \ref{thm:No3s} is the following sharp version of the ampleness theorem of the first author and G.\ Di Cerbo \cite[Thm.\ 1.1]{DiC}, which gave ampleness of $K_X + \alpha D$ for $\alpha \in (\frac{1}{3}, 1)$.

\begin{theorem}\label{sharp}
Let $(X, D)$ be a smooth toroidal compactification of a ball quotient surface. Then $K_{X}+\alpha D$ is ample for all $\alpha\in \big(\frac{1}{4}, 1\big)$. Moreover, this ampleness range is sharp, i.e., there are ball quotient compactifications for which $K_{X}+\frac{1}{4}D$ is nef but not ample.
\end{theorem}

Key in constructing examples where $K_X + \frac{1}{4} D$ is nef but not ample is the existence of ball quotient compactifications for which $X$ is birational to either an Abelian or bielliptic surface. In particular, we exploit the fact that all such known examples contain a holomorphically embedded $\pp^1$ defining a punctured curve on $X \ssm D$ that is totally geodesic in the complex hyperbolic metric. In fact, we will use Theorem \ref{thm:No3s} to show the following, which proves that this phenomenon is no coincidence.

\begin{theorem}\label{thm:AllGeodesic}
Let $(X, D)$ be a ball quotient compactification such that $X$ is birational to either an Abelian or bielliptic surface $Y$. Then $X$ is the blowup of $Y$ at $n$ distinct points for some $n \ge 1$. Every exceptional curve on $X$ meets $D$ transversally in four points and determines a totally geodesic $4$-punctured $\pp^1$ on $X \ssm D$. Conversely, every holomorphically embedded totally geodesic sphere on $X \ssm D$ is a $4$-punctured $\pp^1$ arising from the exceptional locus of the blowup.
\end{theorem}

We briefly recall the history of the existence of ball quotient compactifications birational to an Abelian or bielliptic surface. The first Abelian example was constructed by Hirzebruch \cite{Hir84}, and it is closely related to the arithmetic of the Eisenstein integers. Later, Holzapfel constructed an example based on the Gaussian integers \cite{Hol04}. The first bielliptic ball quotient compactifications were discovered only very recently in \cite{DS15}. These examples emerged in the study of minimal volume complex hyperbolic surfaces with cusps, and they are all modeled on the Eisenstein integers. See \cite{DS16, DS17} for other applications.

Recall that bielliptic surfaces are the minimal projective surfaces with Kodaira dimension $\kappa = 0$ and irregularity $q = 1$. A classical result of Bagnera and de Franchis from 1907 classifies bielliptic surfaces into seven topological types, characterized by a finite group $G$ such that $X$ is the quotient of a product of two elliptic curves by a free action of $G$. Moreover, $G$ must be $\zz/2$, $\zz/3$, $(\zz/3)^2$, $\zz/4$, $\zz/4 \times \zz/2$, $(\zz/2)^2$, or $\zz/6$ (see \S \ref{sec:MainEx} for the precise classification). The examples in \cite{DS15} are associated with the groups $\zz / 3$ and $(\zz / 3)^2$. 

In \S \ref{sec:MainEx} we construct the first examples of bielliptic ball quotient compactifications modeled on the Gaussian integers.  In the Bagnera-de Franchis classification there are exactly four group actions associated with the Gaussian integers, and we construct examples for each of these four topological types. The associated ball quotients are not commensurable with those studied in \cite{DS15}, but are still of relatively small volume. We also discuss our examples in relation with Holzapfel's example of Abelian type \cite{Hol04}. In terms of the classification of bielliptic surfaces, this gives the following.

\begin{theorem}\label{thm:MainExists}
For each of the groups $\zz/2$, $(\zz/2)^2$, $\zz/4$, and $\zz/4 \times \zz/2$, there exists a ball quotient compactification with minimal model a bielliptic surface of that type.
\end{theorem}

We conclude this introduction by addressing some topological aspects connected with the existence of ball quotient compactifications with Kodaira dimension zero. Smooth projective surfaces of Kodaira dimension zero are birational to Abelian surfaces, bielliptic surfaces, K3 surfaces, and Enriques surfaces \cite[Ch.\ VIII]{Bea}. While much is now known about the existence of bielliptic and Abelian ball quotient compactifications, it remains unclear whether or not ball quotient compactifications birational to K3 or Enriques surfaces exist. We note however that one can use known examples of ball quotient compactifications birational to Abelian surfaces to construct ball quotient \emph{orbifolds} whose underlying analytic space admits a compactification by a Kummer K3 surface. For example, see \cite[\S 5.5.4A]{HolzBook}. 

This problem is connected with several topological problems regarding smooth toroidal compactifications. For example, it was previously claimed in the literature that ball quotient compactifications must have nonzero $1^{st}$ Betti number. This would imply the nonexistence of ball quotient compactifications birational to K3 or Enriques surfaces. Unfortunately, this claim is also not true. In the Appendix, we describe a ball quotient compactification of general type with irregularity zero and Euler number $48$, found using Magma \cite{Magma}.

\subsection*{Acknowledgments} The authors thank the International Centre for Theoretical Physics (ICTP) for the excellent working environment during the early stages of this collaboration, and the referee for pertinent comments on the manuscript. The first author would like to thank Maria Beatrice Pozzetti for asking whether or not bielliptic ball quotient compactifications modeled on the Gaussian integers exist, and Gabriele Di Cerbo for asking whether an improved version of Theorem 1.1 in \cite{DiC} could be proved in the case of surfaces. Finally, he would like to thank the Stony Brook University math department for generously encouraging him through the years and for providing an inspiring environment of brilliant math. The first author was partially supported by the S.-S. Chern Fellowship at ICTP. The second author was supported by the National Science Foundation under
Grant Number NSF 1361000 and Grant Number 523197 from the Simons Foundation/SFARI.

\section{Totally geodesic (and other) punctured spheres}\label{sec:Punctures}

To start, we briefly recall the notion of a general ball quotient compactification. Let $X$ be a smooth projective surface with canonical divisor $K_X$ and $D \subset X$ be a reduced simple normal crossings divisor. Then one has the logarithmic Chern numbers
\begin{align*}
\ol{c}_1^2(X, D) &= (K_X + D)^2 \\
\ol{c}_2(X, D) &= e(X \ssm D),
\end{align*}
where $e$ denotes the topological Euler number. When $K_X + D$ is big and nef one has the logarithmic Bogomolov--Miyaoka--Yau inequality $\ol{c}_1^2 \le 3 \ol{c}_2$ (henceforth \emph{log-BMY}). Moreover, equality occurs if and only if $X \ssm D$ is biholomorphic to the quotient $\B^2 / \Gam$, where $\Gam$ is a torsion-free nonuniform lattice  in $\PU(2,1)$ with all parabolic elements rotation-free acting by isometries on the unit ball $\B^2$ in $\cc^2$ equipped with the symmetric Bergman metric. Finally, it can be shown that in this case  $D$ is necessarily a disjoint union of reduced elliptic curves with negative self-intersection. One then says that $X$ is a \emph{smooth toroidal compactification} of $\B^2 / \Gam$ with \emph{compactification divisor} $D$. See \cite{Ash}, \cite{Mok}, or \cite{DS15} for the proofs of these statements and more details. In this paper we will shorten this to \emph{ball quotient compactification}, as we will not consider any other type of compactification.

\medskip

Let $X$ be a smooth projective variety with $D \subset X$ a divisor such that the pair $(X, D)$ is a ball quotient compactification, and let $\om_1$ denote the locally symmetric metric on $X\ssm D$. When regarded as a current on $X$, this K\"ahler metric is proportional to the cohomology class of $K_{X}+D$ \cite{Mumford}.

Consider a smooth embedded irreducible complete curve $C\hookrightarrow X$ not contained in the boundary divisor $D$. We assume that $C$ and $D$ intersect transversally, set
\[
\{p_{1}, \dots, p_{k}\} = C \cap D,
\]
and let $C_{k}$ denote the curve $C$ punctured at the points $p_{1}, \dots, p_{k}$. The Bergman metric $\omega_{1}$ then induces a smooth finite-volume K\"ahler metric $\omega_{C_{k}}$ on the punctured curve $C_{k}$ with associated Ricci form $\rho_{C_{k}}$. By \cite[Prop.\ IX.9.5]{Kobayashi}, we have
\begin{equation}\label{kobayashi}
\rho_{C_{k}}=-(\omega_{1}){|_{C_{k}}}-\mathcal{R},
\end{equation}
where $\mathcal{R}$ is a semidefinite form that vanishes if and only if $C$ is embedded and totally geodesic.

Observe that $\omega_{C_{k}}$ has finite volume. The finite volume Gauss--Bonnet theorem then implies that
\[
-\frac{1}{2\pi}\int_{C_{k}}\rho_{C_{k}}=-\chi(C_{k})=-\chi(C)+k=2g(C)-2+k.
\]
Next, recall that proportionality gives $c_{1}(K_{X}+D)=\frac{3}{4\pi}\omega_{1}$. Thus \eqref{kobayashi} gives
\begin{equation}\label{eq:kobayashi}
-\frac{1}{2\pi}\int_{C_{k}}\rho_{C_{k}}=2g(C)-2+k\ge\frac{1}{2\pi}\int_{C}\omega_{1}=\frac{2}{3}(K_{X}+D)\cdot C
\end{equation}
with equality if and only if $C$ is an embedded totally geodesic curve. Thus we have:

\begin{proposition}\label{bound}
Let $C$ be an embedded curve in a ball quotient surface compactification $(X, D)$ such that $C$ intersects $D$ transversally in $k$ points. Then
\[
3(g(C)-1)+\frac{k}{2}\ge K_{X}\cdot C
\]
with equality if and only if $C_{k}$ is totally geodesic in $X\ssm D$ with respect to its locally symmetric metric $\omega_{1}$.
\end{proposition}

Proposition \ref{bound} can be alternatively formulated as follows.

\begin{proposition}\label{bound2}
Let $C$ be an embedded curve in a ball quotient surface compactification $(X, D)$ such that $C$ intersects $D$ transversally. We then have 
\[
3 C^{2}\ge -(K_{X}+D)\cdot C
\] 
with equality if and only if $C \cap (X \ssm D)$ is totally geodesic in $X\ssm D$ with respect to $\omega_{1}$.
\end{proposition}

\begin{proof}
First, we have that
\[
g(C)-1=\frac{K_{X}\cdot C+C^{2}}{2}.
\]
Next, we have that $k=D\cdot C$. Thus, the inequality given in Proposition \ref{bound} becomes
\[
\frac{3}{2}(K_{X}\cdot C+C^{2})+\frac{D\cdot C}{2}\ge K_{X}\cdot C,
\]
hence
\[
3C^{2}\ge -K_{X}\cdot C-D\cdot C,
\]
which completes the proof.
\end{proof}

This gives us the following consequence when $C$ is an exceptional curve on $X$.

\begin{corollary}\label{exceptional}
An exceptional curve of the first kind in a ball quotient compactification $X$ that intersects $D$ transversally must have at least four intersection points.
\end{corollary}

\begin{proof}
Let us denote the exceptional curve of the first kind in $X$ by $C$. Recall that $C\simeq\pp^{1}$ and $K_{X}\cdot C=C^{2}=-1$. Next, observe that $C$ must intersect $D$ at least three times, as $X\ssm D$ has negative curvature. Using Proposition \ref{bound}, we obtain $k\ge 4$ with equality if and only if $C_{k}$ is totally geodesic in $X\ssm D$ with respect to the Bergman metric.
\end{proof}

We now show that Corollary \ref{exceptional} holds even without the transversality requirement. More precisely, let $C$ be an exceptional curve of the first kind in the smooth compactification $(X, D)$. Then, there are integers $m_i \ge 1$ such that
\[
D\cap C = \sum^{k}_{i=1}m_{i}p_{i}.
\]
Then $m_i = 1$ if and only if $C$ meets $D$ transversally at $p_i$, and $m_i \ge 2$ otherwise. The generalized relative Hirzebruch--H\"ofer proportionality theorem given in Theorem 0.1 of \cite{Muller} is
\begin{equation}\label{relative}
3 C^{2}\ge -K_{X}\cdot C +\sum^{k}_{i=1}(2m_{i}-3)
\end{equation}
with equality if and only if $C\ssm \{p_1, ..., p_k \}$ is totally geodesic with respect to the Bergman metric on $X\ssm D$.

Therefore, if $C$ is an exceptional curve of the first kind we have
\[
 C^{2}=K_{X}\cdot C=-1,
\]
and then
\[
-4\ge \sum^{k}_{i=1}(2m_{i}-3),
\]
with equality if and only if $C\ssm \{p_1, ..., p_k \}$ is totally geodesic. We then have the following generalization of Corollary \ref{exceptional}.

\begin{proposition}\label{exceptional2}
An exceptional curve of the first kind in a ball quotient compactification $X$ must have at least four intersection points with the compactification divisor $D$. Moreover, it has exactly four intersection points if and only if it is totally geodesic.
\end{proposition}

We are now ready to prove Theorem \ref{thm:No3s}.

\begin{proof}[Proof of Theorem \ref{thm:No3s}]
We first prove (1). Suppose that $C_0$ is a holomorphically immersed totally geodesic submanifold of $X \ssm D$. When $C_0$ is embedded, \cite[Lem.\ 4.4.8]{HolzBook} gives that the compactification $C$ of $C_0$ in $X$ is smooth and that $C$ intersects $D$ transversally. The result is then immediate from Proposition \ref{bound}. Indeed, we have
\[
3(g(C)-1)+\frac{k}{2} = K_{X}\cdot C
\]
with $3(g(C)-1)$ and $K_{X}\cdot C$ integers, hence $k$ is even.

We now argue that we can pass to an \'etale finite cover $X^\prime \ssm D^\prime$ of $X \ssm D$ such that $C_0$ lifts to an embedded totally geodesic $C_0^\prime$ in $X^\prime \ssm D^\prime$ with $k$ cusps. Then $X^\prime$ is a ball quotient compactification and $C_0^\prime$ is a $k$-cusped embedded totally geodesic $k$-punctured $\pp^1$ in $X^\prime \ssm D^\prime$, which we have seen is impossible when $k$ is odd.

Indeed, it is well-known that totally geodesic immersions can be promoted to embeddings in finite covers. We sketch the argument. Let $C_0^\prime$ denote the $k$-punctured $\pp^1$ so that there is a totally geodesic immersion $f : C_0^\prime \to X \ssm D$ with image $C_0$, and let $\widetilde{f} : \B^1 \hookrightarrow \B^2$ be the associated totally geodesic embedding of universal coverings. Suppose
\[
\Gam = \pi_1(X \ssm D) \subset \PU(2,1)
\]
and let $\Del \subset \Gam$ be the stabilizer in $\Gam$ of $\widetilde{f}(\B^1)$. Then $\widetilde{f}(\B^1) / \Del$ is homotopy equivalent to $C_0^\prime$. Now, $\Del$ is \emph{separable} in $\Gam$, i.e., given $\gamma \in \Gam$ with $\gamma \notin \Del$, there exists a finite index subgroup $\Gam^\prime \subset \Gam$ such that $\Del \subset \Gam^\prime$ and $\gamma \notin \Gam^\prime$. See the ``Lemme Principal'' on p.\ 113 of \cite{Ber00}.

Next, we apply the classic topological characterization of separability due to Peter Scott \cite{Scott} to obtain a finite \'etale covering $X^\prime \ssm D^\prime$ of $X \ssm D$ into which $C_0^\prime$ embeds as a holomorphic totally geodesic submanifold. Replace $(X, D)$ with $(X^\prime, D^\prime)$ and $C_0$ with its lift to $X^\prime \ssm D^\prime$, which is biholomorphic to $C_0^\prime$. Crucially, $C_0$ and $C_0^\prime$ have the same number of cusps $k$. Then the application of Proposition \ref{bound} that began the proof now applies to $C_0^\prime$, and hence $k$ cannot be odd.

Now, we prove (2). We must show that a ball quotient compactification contains no holomorphically embedded $3$-punctured $\pp^1$ arising from a smooth rational curve on $X$. Let $C$ be such a rational curve. Applying the adjunction formula to Equation \eqref{relative} we have
\begin{equation}\label{eq:MSrational}
2 C^2 - 2 \ge \sum_{i=1}^3(2m_i-3) \ge -3,
\end{equation}
so that we necessarily have $C^{2}\geq 0$.

Then $K_X \cdot C < 0$ by adjunction, hence $K_X$ is not nef. Since $C$ is not a $-1$ curve, it follows that the plurigenera of $X$ are all zero, hence $X$ is a ruled surface. If $C^2 \ge 1$, then $C$ cannot be contained in a fiber of the ruling, so $X$ must be a rational ruled surface. Case-by-case analysis shows that the only rational ruled surface with a rational curve of positive self-intersection is $\pp^2$. Clearly $\pp^2$ is not a ball quotient compactification, since elliptic curves on $\pp^2$ have positive self-intersection.

Next, suppose that $C^2 = 0$. We then have
\[
-2 \ge \sum_{i=1}^3(2m_i-3),
\]
and the only possibility is $m_1 = m_2 = m_3 = 1$, so $C$ meets $D$ transversally in three points. We already saw that $X$ must be ruled over a curve $B$ and, since $X$ contains an elliptic curve, the Hurwitz formula implies that $B$ is rational or elliptic. Moreover, Zariski's lemma implies that $C$ must be a smooth fiber of the ruling.

We can rule out minimal ruled surfaces as in the proof of Theorem 2.3 in \cite{DiC}. We briefly sketch the argument. If $X$ is a Hirzebruch surface, then all irreducible curves of negative self-intersection are rational, which is impossible since $X$ contains elliptic curves of negative self-intersection. If $X$ is ruled over an elliptic curve, then $D$ must consist of exactly one zero section of the ruling. One then finds a once-punctured $\pp^1$ in $X \ssm D$, which is impossible since $X \ssm D$ admits a K\"ahler metric of negative curvature.

In other words, $X$ must be the blow up of a surface ruled over an elliptic or rational curve, hence there exists a fiber that contains at least one exceptional curve of the first kind. On the other hand, Proposition \ref{exceptional2} shows that the compactifying divisor must meet an exceptional curve of the first kind in at least four distinct points. Note that no component of the compactifying divisor can be contained in a fiber. Now, let
\[
F=E_{1}+\dots+E_{j}
\]
be a fiber containing at least one exceptional curve, which we assume is $E_{1}$. Fibers are numerically equivalent and the compactification divisor satisfies $D\cdot C=3$. We then have the contradiction
\[
3=D\cdot C=D\cdot F\geq D\cdot E_{1}\geq 4,
\]
which proves (2).

Finally we prove (3) and (4). Again by \cite[Lem.\ 4.4.8]{HolzBook} we know that the compactification $C$ of the totally geodesic punctured $\pp^{1}$ in $X$ is smooth and intersects $D$ transversally. If such a $\pp^{1}$ is punctured 4 times, then Proposition \ref{bound} implies that
\[
-3+2=K_{X}\cdot C,
\]
so $K_{X}\cdot C=C^{2}=-1$. This shows that $C$ is indeed an exceptional divisor in $X$. Similarly, if $C$ is the compactification of a totally geodesic 6-punctured sphere in $X$, Proposition \ref{bound} implies that
\[
-3+3=K_{X}\cdot C,
\]
so $C^{2}=-2$ by adjuction. This completes the proof of the theorem.
\end{proof}

We can now prove Theorem \ref{sharp}. First we briefly recall the notion of an \emph{extremal ray} in the cone $\ol{\mathrm{NE}}(X)$ of numerically effective $1$-cycles on $X$, along with its length. A ray $R$ in $\ol{\mathrm{NE}}(X)$ is \emph{extremal} when $x+y \in R$ and $x,y \in \ol{\mathrm{NE}}(X)$ implies that $x,y \in R$. Suppose $R$ is a ray such that $K_X \cdot Z < 0$ for any effective $1$-cycle $Z \in R$. Then the \emph{length} of $R$ is the minimum of $-K_X \cdot C$ for $C$ a rational curve in $R$.

\begin{proof}[Proof of Theorem \ref{sharp}]
The proof is a combination of Mori's theory together with the generalized relative proportionality principle. More precisely, Theorem 2.3 in \cite{DiC} implies that the length of an extremal ray in a smooth toroidal compactification of a ball quotient is at most one. Moreover, the associated contraction map is the blowup map at a codimension two smooth subvariety. See Corollary 2.8 in \cite{DiC}.

Therefore, in the case of surfaces we conclude that the extremal rays in a smooth ball quotient compactification with non-nef canonical divisor are exceptional curves of the first kind, that is, smooth rational curves $\{C_{i}\}$ in $X$ such that $K_{X}\cdot C_{i}=C_{i}^2=-1$. By the cone theorem \cite[Thm.\ 10-2-1]{Matsuki}, we have that
\begin{align}\label{cone}
\ol{\mathrm{NE}}(X)=\ol{\mathrm{NE}}(X)_{K_{X}\ge 0}+\sum \rr_{\ge 0}[C_{i}],
\end{align}
where $\ol{\mathrm{NE}}(X)_{K_{X}\ge 0}$ are the points in $\ol{\mathrm{NE}}(X)$ that pair nonnegatively with $K_{X}$, and where the $\{C_{i}\}$ are the (possibly countably many) extremal rays. See \cite[\S 2]{DiC} for further details and references.

Then any curve $C$ on $X$ is numerically equivalent to
\[
a_{1}C_{1}+...+a_{k}C_{k}+F,
\]
where the $a_{i}$ are positive real numbers, the $\{C_{i}\}$ are a finite collection of extremal rays, and $F$ is an effective divisor such that $K_{X}\cdot F\ge 0$. Thus, by Proposition \ref{exceptional2}, we conclude that $D \cdot C_{i}\ge 4$ for all $i$, with equality if and only if the punctured sphere defined by $C_{i}$ on $X \ssm D$ is totally geodesic.

Since $K_{X}+D$ is big and nef for any smooth ball quotient compactification, we also have that
\[
(K_{X}+\alpha D)\cdot C\ge \sum^{k}_{i=1}a_{i}(K_{X}+\alpha D)C_{i}\ge\sum^{k}_{i=1}a_{i}(-1+4\alpha), 
\]
for any curve $C$ in $X$, since $F$ is effective. Thus we also have that
\[
(K_{X}+\alpha D)\cdot C>0
\]
for any $\alpha\in (\frac{1}{4}, 1)$. This fact combined with the fact that $K_{X}+D$ is big immediately implies that $K_{X}+\alpha D$ is indeed an ample $\rr$-divisor for any $\alpha\in(\frac{1}{4}, 1)$. Moreover, if $K_{X}+\frac{1}{4}D$ is nef but not ample then there must exist an exceptional curve of the first kind $C_{i}$ such that $D\cdot C_{i}=4$. By Proposition \ref{exceptional2}, we have that such a curve $C_{i}$ determines a totally geodesic 4-punctured $\pp^{1}$ in $X\ssm D$.

In \S \ref{sec:MainEx}, we construct ball quotient compactifications with exceptional curves of the first kind intersecting the boundary divisor transversally in exactly four points. See \cite{Hir84, Holzapfel, Hol04, DS15} for other examples. Pick one such ball quotient $(X, D)$. Then we have 
\[
\left(K_{X}+\frac{1}{4}D\right)\cdot E_{i}=0
\]
for any of the exceptional divisors $\{E_{i}\}$ in $X$, so these exceptional divisors give totally geodesic 4-punctured $\pp^1$. Clearly $K_{X}+\frac{1}{4}D$ is not ample. In conclusion, the ampleness range given by $\alpha\in (\frac{1}{4}, 1)$ cannot be improved in general and the result is indeed sharp as claimed.
\end{proof}

We close this section with the proof of Theorem \ref{thm:AllGeodesic}.

\begin{proof}[Proof of Theorem \ref{thm:AllGeodesic}]
Let $(X, D)$ be a ball quotient compactification such that $X$ is birational to an Abelian or bielliptic surface $Y$. Then there is sequence of blowups $\pi : X \to Y$. Moreover, $K_X$ is numerically equivalent to the blowup divisor
\begin{equation}\label{eq:Canonical}
E = \sum_{i = 1}^n \alpha_i E_i,
\end{equation}
where the $E_i$ are distinct smoothly embedded rational curves and each $\alpha_i \ge 1$ is an integer. Since Abelian and bielliptic surfaces have universal cover $\cc^2$, every rational curve on $X$ is numerically equivalent to some $E_i$.

We then have
\begin{align*}
\ol{c}_1^2(X, D) &= -n - D^2 \\
3\, \ol{c}_2(X, D) &= 3n,
\end{align*}
so the log-BMY inequality gives $-D^2 = 4 n$. However, adjunction gives
\[
-D^2 = K_X \cdot D = \sum_{i = 1}^n \alpha_i D \cdot E_i.
\]
Theorem \ref{thm:No3s} then implies that $D \cdot E_i \ge 4$ for each $i$, hence $-D^2 \ge 4 n$ with equality if and only if $D \cdot E_i = 4$ and $\alpha_i = 1$ for all $i$. This proves that every $E_i$ defines a $4$-punctured $\pp^1$ on $X$, and the associated punctured curve on $X \ssm D$ is therefore totally geodesic.

From our conclusion that $\alpha_i = 1$ for all $i$, we claim that $X$ must be the blowup of $Y$ at $n$ distinct points, with the $E_i$ precisely the exceptional curves on $X$ of the first kind. Indeed, suppose that $\pi$ can be realized by the sequence of blowups
\[
X_n = Y \to X_{n-1} \to \cdots \to X_1 \to X_0 = Y
\]
with $\widehat{E}_j$ the exceptional divisor of the blowup $\sig_i : X_j \to X_{j-1}$. We then have the canonical divisor formula
\[
K_{X_i} = \sig_i^* K_{X_{i-1}} + \widehat{E}_i,
\]
where $\sig_i^*$ denotes the \emph{total} transform. Since $K_Y$ is trivial, it is immediate that $K_X = \sum E_i$ when $X$ is the blowup of $Y$ at $n$ distinct points.

On the other hand, suppose that
\[
K_{X_{j-1}} = \sum_{k = 1}^{j-1} \beta_k \widetilde{E}_k
\]
and that $\sig_j$ is the blowup of $X_{j-1}$ at a point on $\widetilde{E}_k$, $1 \le k \le j-1$. Then we have
\begin{align*}
K_{X_j} &= \sig^* K_{X_{j-1}} + \widehat{E}_j \\
&= 2 \widehat{E}_j + \sum_{k = 1}^{j-1} \beta_k \widetilde{E}_k^\prime,
\end{align*}
where $\widetilde{E}_k^\prime$ is the strict transform of $\widetilde{E}_k$. Indeed, the total transform of $\widetilde{E}_k$ is $\widetilde{E}_k^\prime + \widehat{E}_j$. Inducting on the length of the chain of blowups from $X$ to $Y$, it follows that $\alpha_j \ge 2$ in \eqref{eq:Canonical}, and that $\alpha_j = 1$ for all $j$ if and only if $X$ is the blowup of $Y$ at $n$ distinct points. This proves the claim, and therefore completes the proof of the theorem.
\end{proof}

\section{New bielliptic compactifications}\label{sec:MainEx}

In this section, we construct our main new examples of a bielliptic ball quotient compactifications. First we briefly recall the classification of bielliptic surfaces.

One characterization is that bielliptic surfaces are the minimal smooth projective surfaces with Kodaira dimension $0$ and irregularity $1$ (i.e., $1^{st}$ Betti number $2$). See \cite[Ch.\ VIII]{Bea}. By the classical theorem of Bagnera and de Franchis \cite[Thm.\ VI.20]{Bea}, a bielliptic surface is constructed as follows. Let $E_1 \times E_2$ be a product of two elliptic curves and $G$ be a group of translations of $E_2$ that also acts on $E_1$ as a group of automorphisms with quotient $\pp^1$. Then $X = (E_1 \times E_2) / G$ is a bielliptic surface. More specifically:

\begin{theorem}[Bagnera--de Franchis, 1907]\label{bagnera}
Let $E_{\lambda}$ and $E_{\tau}$ be elliptic curves associated with the lattices $\zz[1, \lambda]$ and $\zz[1, \tau]$, respectively, and G be a group of translations of $E_{\tau}$ acting on $E_{\lambda}$ such that $E_{\lambda}/G=\pp^{1}$. Then every bielliptic surface is of the form $(E_{\lambda}\times E_{\tau})/G$ where $G$ has one of the following types:
\begin{enumerate}

\item $G=\zz / 2$ acting on $E_{\lambda}$ by $x\rightarrow -x$;

\item $G=\zz / 2 \times\zz / 2$ acting on $E_{\lambda}$ by
\[
x\rightarrow -x \quad\textrm{and}\quad x\rightarrow x+\alpha_{2},
\]
where $\alpha_{2}$ is a $2$-torsion point;

\item $G=\zz / 4$ acting on $E_\lambda$ by $x\rightarrow \lambda x$, where $\lambda = i$;

\item $G=\zz / 4\times \zz / 2$ acting on $E_\lambda$ by
\[
x\rightarrow \lambda x \quad\textrm{and}\quad x\rightarrow x+\frac{1+\lambda}{2},
\]
where $\lambda = i$;

\item $G=\zz / 3$ acting on $E_{\lambda}$ by $x\rightarrow \lambda x$, where $\lambda=e^{\frac{2\pi i}{3}}$;

\item $G=\zz / 3\times \zz / 3$ acting on $E_{\lambda}$ by
\[
x\rightarrow \lambda x \quad\textrm{and}\quad x\rightarrow x+\frac{1-\lambda}{3},
\]
where $\lambda=e^{\frac{2\pi i}{3}}$;

\item $G=\zz / 6$ acting on $E_{\lambda}$ by $x\rightarrow \zeta x$, where $\lambda=e^{\frac{2\pi i}{3}}$ and $\zeta=e^{\frac{\pi i}{3}}$.

\end{enumerate}
\end{theorem}

\medskip

\noindent
\textbf{Realizing $\zz/2$.}

\medskip

To begin our construction, let $E$ be the elliptic curve associated with the Gaussian lattice $\zz[1, i]=\zz+\zz i$, where $i^2 = -1$. In this section, we construct a $\zz/2\zz$ bielliptic ball quotient modeled on the Abelian surface $A=E \times E$. To this aim, consider the degree two automorphism $\varphi: A\rightarrow A$ given by
\begin{equation}\label{twisted}
\varphi(w, z)=\left(-w+\frac{1+i}{2}\,,\, z+\frac{1+i}{2}\right).
\end{equation}
Note that $\varphi$ generates a free action of $\zz/2$ on $A$. Let
\[
\pi: A\rightarrow B = A / \langle \varphi \rangle
\]
be the associated degree two \'etale quotient. This is a slightly different version of the ``standard'' $\zz/2$-action appearing in the Bagnera--de Franchis classification. We adopt this one for computational convenience, and discuss the connection with the standard action below. In what follows, $(w,z)$ will denote coordinates on $A$ and $[w,z]$ the coordinates on $B$ for $\pi(w,z)$.

Next, we define elliptic curves
\begin{align*}
E_{1}&=(z, z) & E_{2}&=\left(iz+\frac{1+i}{2}, z\right) \\
E_{3}&=(-z, z) & E_{4}&=\left(-iz+\frac{1+i}{2}, z\right)
\end{align*}
on the Abelian surface $A$. These curves are constructed so that
\begin{equation}\label{involution1}
\varphi(E_{1})=E_{3}\quad \quad \varphi(E_{2})=E_{4}.
\end{equation}
Moreover, one can easily check that:
\begin{align*}
E_{1} \cap E_{3}&=\left\{(0, 0), \left(\frac{1}{2}, \frac{1}{2}\right), \left(\frac{i}{2}, \frac{i}{2}\right), \left(\frac{1+i}{2}, \frac{1+i}{2}\right)\right\}\\
E_{2} \cap E_{4}&=\left\{\left(0, \frac{1+i}{2}\right), \left(\frac{1}{2}, \frac{1}{2}\right), \left(\frac{i}{2}, \frac{i}{2}\right), \left(\frac{1+i}{2}, 0\right)\right\}\\
E_{1} \cap E_{2}&=E_{1} \cap E_{4}= E_{2} \cap E_{3}=E_{3} \cap E_{4}=\left\{\left(\frac{1}{2}, \frac{1}{2}\right), \left(\frac{i}{2}, \frac{i}{2}\right)\right\}
\end{align*}
Furthermore, consider the two vertical and two horizontal elliptic curves in $A$ defined by the equations:
\begin{align*}
F_{1}&=\big(w, 0\big) & F_{3}&=\big(0, z\big) \\
F_{2}&=\left(w, \frac{1+i}{2}\right) & F_{4}&=\left(\frac{1+i}{2}, z\right)
\end{align*}
These curves satisfy
\begin{equation}\label{involution2}
\varphi(F_{1})=F_{2}\quad \quad \varphi(F_{3})=F_{4},
\end{equation}
and we easily check that:
\begin{align*}
F_{1} \cap F_{3}=E_{1} \cap F_{1}=E_{1} \cap F_{3}=E_{3} \cap F_{1}=E_{3} \cap F_{3}&=\left\{(0, 0)\right\} \\
F_{1} \cap F_{4}=E_{2} \cap F_{1}=E_{2} \cap F_{4}=E_{4} \cap F_{1}=E_{4} \cap F_{4}&=\left\{\left(\frac{1+i}{2}, 0\right)\right\} \\
F_{2} \cap F_{3}=E_{2} \cap F_{2}=E_{2} \cap F_{3}=E_{4} \cap F_{2}=E_{4} \cap F_{3}&=\left\{\left(0, \frac{1+i}{2}\right)\right\} \\
F_{2} \cap F_{4}=E_{1} \cap F_{2}=E_{1} \cap F_{4}=E_{3} \cap F_{2}=E_{3} \cap F_{4}&=\left\{\left(\frac{1+i}{2}, \frac{1+i}{2}\right)\right\}
\end{align*}
See Figure \ref{fig:AbelianArrangement}.
\begin{figure}
\begin{center}
\begin{tikzpicture}
\useasboundingbox (-2,-1.5) rectangle (7.5,7.5);
\draw [very thick, yellow!30, orange!70, rounded corners] (-0.5, 7) -- (0,6) -- (2,2) .. controls (4,-2) and (2, 9) .. (4,4) -- (6,0) -- (6.5, -1);
\draw [very thick, yellow!30, orange!70, rounded corners] (-1, 6.5) -- (0,6) .. controls (1,5) and (1,5) .. (2,2) .. controls (4.3,-4) and (2,8) .. (4, 4) .. controls (5,1) and (5,1) .. (6,0) -- (7, -0.5);
\draw [line width=4pt, white] (-1,-1) -- (7,7);
\draw [line width=4pt, white, rounded corners] (-1, -0.5) -- (0,0) -- (1,0.5) -- (2,2) -- (3,3.5) -- (4,4) -- (5, 4.5) -- (6,6) -- (6.5, 7);
\draw [very thick, green] (-1,-1) -- (7,7);
\draw [very thick, green, rounded corners] (-1, -0.5) -- (0,0) -- (1,0.5) -- (2,2) -- (3,3.5) -- (4,4) -- (5, 4.5) -- (6,6) -- (6.5, 7);
\draw [very thick, red] (-1,0) -- (7,0);
\draw [very thick, red] (-1,6) -- (7,6);
\draw [very thick, blue] (0,-1) -- (0,7);
\draw [very thick, blue] (6,-1) -- (6,7);
\draw (-1.25,0) node [red] {$F_1$};
\draw (-1.25,6) node [red] {$F_2$};
\draw (0, -1.25) node [blue] {$F_3$};
\draw (6, -1.25) node [blue] {$F_4$};
\draw (-1.25,-1.25) node [green] {$E_1$};
\draw (6.5,7.25) node [green] {$E_3$};
\draw (-0.5,7.25) node [yellow!30, orange!70] {$E_2$};
\draw (6.75,-1) node [yellow!30, orange!70] {$E_4$};
\draw (0.4, -0.25) node [font=\footnotesize] {$(0,0)$};
\draw (6.8, 5.65) node [font=\footnotesize] {$\left(\frac{1+i}{2},\frac{1+i}{2}\right)$};
\draw (0.6, 6.35) node [font=\footnotesize] {$\left(0,\frac{1+i}{2}\right)$};
\draw (5.4, -0.4) node [font=\footnotesize] {$\left(\frac{1+i}{2}, 0\right)$};
\draw (1.45, 2.05) node [font=\footnotesize] {$\left(\frac{1}{2}, \frac{1}{2}\right)$};
\draw (4.6, 3.9) node [font=\footnotesize] {$\left(\frac{i}{2}, \frac{i}{2}\right)$};
\filldraw (0,0) circle [black, radius=0.075cm];
\filldraw (0,6) circle [black, radius=0.075cm];
\filldraw (6,0) circle [black, radius=0.075cm];
\filldraw (6,6) circle [black, radius=0.075cm];
\filldraw (2,2) circle [black, radius=0.075cm];
\filldraw (4,4) circle [black, radius=0.075cm];
\end{tikzpicture}
\caption{The arrangement of curves on the Abelian surface.}\label{fig:AbelianArrangement}
\end{center}
\end{figure}
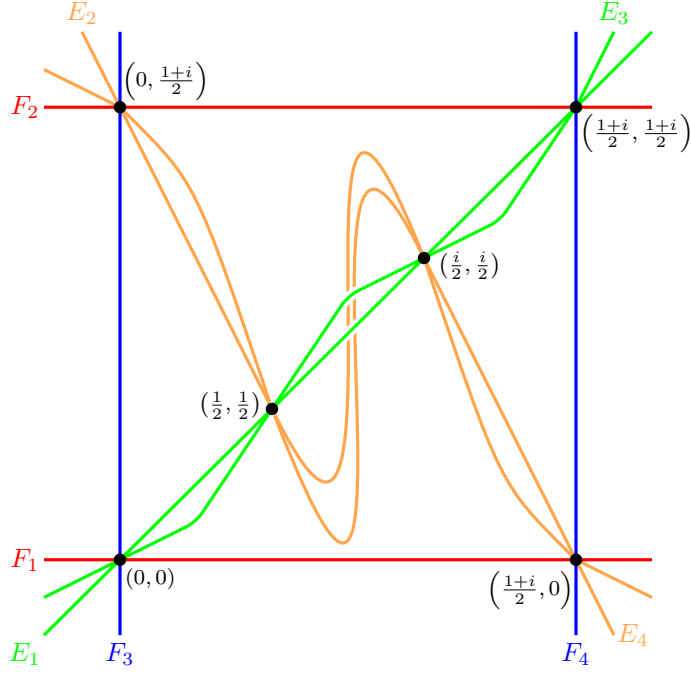
Concluding, the curves $\{E_{j}\}^{4}_{j=1}$ and $\{F_{j}\}^{4}_{j=1}$ meet only in the following six points:
\[
\mathcal{P} = \left\{(0, 0)\,,\, \left(\frac{1}{2}, \frac{1}{2}\right)\,,\, \left(\frac{i}{2}, \frac{i}{2}\right)\,,\, \left(\frac{1+i}{2}, 0\right)\,,\, \left(0, \frac{1+i}{2}\right)\,,\, \left(\frac{1+i}{2}, \frac{1+i}{2}\right) \right\}
\]
Moreover, there are exactly four elliptic curves passing through each of these points.

Let be $Y$ the blowup of $A$ at the points in $\mathcal{P}$. Further, let $\{C_{j}\}^{4}_{j=1}$ and $\{D_{j}\}^{4}_{j=1}$ denote the proper transforms of the elliptic curves $\{E_{j}\}^{4}_{j=1}$ and $\{F_{j}\}^{4}_{j=1}$ to $Y$. Finally, set
\[
\mathcal{D}_0 = \sum_{j = 1}^4(C_j + D_j) \subset Y.
\]
We now have the following:

\begin{lemma}\label{lem:AbelianEx}
The pair $(Y, \mathcal{D}_0)$ is the smooth toroidal compactification of a complex hyperbolic surface with eight cusps and Euler number six.
\end{lemma}

\begin{proof}
One easily computes that
\begin{align*}
\overline{c}^{2}_{1}(Y, \mathcal{D}_0)&= K_Y^2 - \mathcal{D}_0^2 \\
&=-6+4+4+4+4+2+2+2+2 \\
&=18,
\end{align*}
and clearly $e(Y \ssm \mathcal{D}_0) = 6$, since $Y$ is the blowup of an Abelian surface at six points. Therefore $\overline{c}^{2}_{1}(Y, \mathcal{D}_0)=18=3\, \overline{c}_{2}(Y, \mathcal{D}_0)$, hence we have equality in the log-BMY inequality. Since $K_Y + \mathcal{D}_0$ is visibly nef and big, the lemma follows.
\end{proof}

We are now ready to construct the bielliptic example. Define the following curves on the bielliptic surface $B$:
\begin{align*}
G_1 &= \pi(E_1) = \pi(E_3) & G_2 &= \pi(E_2) = \pi(E_4) \\
H_1 &= \pi(F_1) = \pi(F_2) & H_2 &= \pi(F_3) = \pi(F_4)
\end{align*}
Note that this is well defined by \eqref{involution1} and \eqref{involution2}. We then have the following lemma.

\begin{lemma}\label{bquotient}
The curves $G_1, G_2$ are singular elliptic curves with exactly two nodes, and $H_1, H_2$ are smooth elliptic curves. Moreover, we have the following intersections:
\[
\begin{cases} & G_{1} \cap G_{2}=\left\{\left[\frac{1}{2}, \frac{1}{2}\right]\right\} \\
& H_{1} \cap H_{2}=\left\{[0, 0], \left[\frac{1+i}{2}, 0\right]\right\} \\
& G_{1} \cap H_{1}=G_{1} \cap H_{2}=\left\{[0, 0]\right\} \\
& G_{2} \cap H_{1}=G_{2} \cap H_{2}=\left\{\left[\frac{1+i}{2}, 0\right]\right\} \\
\end{cases}
\]
\end{lemma}

\begin{proof}
The curves $E_{1}$ and $E_{3}$ meet transversally in four points, and these points are the disjoint union of two orbits under the action of $\varphi$. Thus, the irreducible curve $G_{1}$ has exactly two nodal singularities at the points $[0, 0]$ and $\left[\frac{1}{2}, \frac{1}{2}\right]$. Similarly, $G_{2}$ has exactly two nodal singularities at the points $\left[\frac{1+i}{2}, 0\right]$ and $\left[\frac{1}{2}, \frac{1}{2}\right]$. Finally, the intersection points of the smooth curves $H_1, H_2$ are easily computed.
\end{proof}

In conclusion, the curves $\{G_{j}\}^{2}_{j=1}$ and $\{H_{j}\}^{2}_{j=1}$ meet only in the following three points:
\begin{align}\label{3intersections}
\left\{[0, 0]\,,\, \left[\frac{1+i}{2}, 0\right]\,,\, \left[\frac{1}{2}, \frac{1}{2}\right]\right\}.
\end{align}
See Figure \ref{fig:BiellipticArrangement}.
\begin{figure}
\begin{center}
\begin{tikzpicture}
\draw [very thick, red] (0,-1) -- (0,7);
\draw [very thick, blue, rounded corners] (-1,-1) -- (3,3) -- (-1,7);
\draw [very thick, green, rounded corners] (0.5,-1) -- (-1,2) -- (-3.5,3.25) -- (-3.5, 3.5) -- (-3.25, 3.5) -- (-2,1) -- (1,-0.5);
\draw [very thick, yellow!30, orange!70, rounded corners] (0.5,7) -- (-1,4) -- (-3.5,2.75) -- (-3.5,2.5) -- (-3.25,2.5) -- (-2,5) -- (1,6.5);
\draw (0,-1.25) node [red] {$H_1$};
\draw (-1.15, -1.25) node [blue] {$H_2$};
\draw (1.3, -0.5) node [green] {$G_1$};
\draw (1.3, 6.5) node [yellow!30, orange!70] {$G_2$};
\draw (0.6, 0.1) node [font=\footnotesize] {$[0,0]$};
\draw (0.9, 5.9) node [font=\footnotesize] {$\left[\frac{1+i}{2},0\right]$};
\draw (-2.15, 3) node [font=\footnotesize] {$\left[\frac{1}{2},\frac{1}{2}\right]$};
\filldraw (0,0) circle [black, radius=0.075cm];
\filldraw (0,6) circle [black, radius=0.075cm];
\filldraw (-3,3) circle [black, radius=0.075cm];
\end{tikzpicture}
\caption{The arrangement of curves on the bielliptic surface.}\label{fig:BiellipticArrangement}
\end{center}
\end{figure}
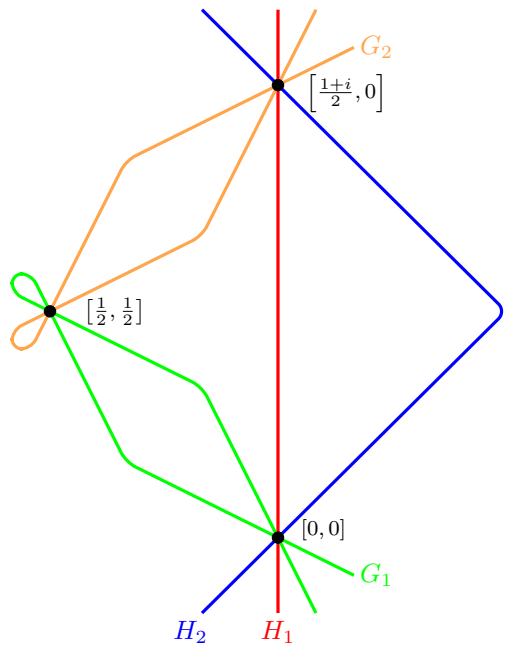
Let $X$ denote the blowup of $B$ at the three points in \eqref{3intersections}, $A_j$ be the proper transform in $X$ of the singular elliptic curve $G_j$, and $B_j$ be the proper transform of the smooth elliptic curve $H_j$ ($j=1,2$). Note that $A_j$ is a smooth elliptic curve. Finally, let $\mathcal{D}$ be the divisor on $X$ determined by these four elliptic curves.

\begin{lemma}\label{example3}
The pair $(X, \mathcal{D})$ is the smooth toroidal compactification of a complex hyperbolic surface with four cusps and Euler number three.
\end{lemma}

\begin{proof}
Exactly as in the proof of Lemma \ref{lem:AbelianEx}, we see that
\[
\overline{c}^{2}_{1}(X, \mathcal{D}) = 9 = 3\, \overline{c}_{2}(X, \mathcal{D})
\]
with $K_X + \mathcal{D}$ nef and big, so the lemma follows.
\end{proof}


Thus $X \ssm \mathcal{D}$ is a ball quotient with smooth toroidal compactification a $\zz/2$ bielliptic surface.

\medskip

\noindent
\textbf{Relationship with the standard bielliptic involution.}

\medskip

We now briefly describe how to write the above example in terms of the \emph{standard} bielliptic involution $\varphi_{1}: A\rightarrow A$ given by
\begin{equation}\label{standard}
\varphi_1(w, z)=\left(-w\,,\, z+\frac{1+i}{2}\right).
\end{equation}
To start, consider the automorphism
\begin{align*}
\psi&: E \times E \rightarrow E \times E \\
\psi(w, z) &=\left(iw+\frac{1+3i}{4}\,,\, z+\frac{3+i}{4}\right),
\end{align*}
and observe that
\begin{align*}
(\psi\circ\varphi)(w, z)&=\left(-iw+\frac{3+i}{4}\,,\, z +\frac{1+3i}{4}\right) \\
(\varphi_1\circ\psi)(w, z)&=\left(-iw+\frac{3+i}{4}\,,\, z+\frac{1+3i}{4}\right).
\end{align*}
We conclude that $\psi\circ\varphi\circ\psi^{-1}=\varphi_{1}$, and hence $\psi$ descends to a map $\psi_1: B\rightarrow B$.

One easily checks that:
\begin{align*}
\psi(E_{1})&=\left(iz+\frac{1}{2}, z\right) & \psi(F_{1})&=\left(w, \frac{3+i}{4}\right) \\
\psi(E_{2})&=\left(-iz+\frac{i}{2}, z\right) & \psi(F_{2})&=\left(w, \frac{1+3i}{4}\right) \\
\psi(E_{3})&=\left(-z+\frac{1+i}{2}, z\right) & \psi(F_{3})&=\left(\frac{1+3i}{4}, z\right) \\
\psi(E_{4})&=(z,z) & \psi(F_{4})&=\left(\frac{3+i}{4}, z\right)
\end{align*}
The intersections between these curves are at
\[
\left\{\left(\frac{1+i}{4}, \frac{1+i}{4}\right)\,,\, \left(\frac{1+3i}{4}, \frac{1+3i}{4}\right)\,,\, \left(\frac{1+3i}{4}, \frac{3+i}{4}\right)\,,\, \left(\frac{3+i}{4}, \frac{1+3i}{4}\right),\right.
\]
\[
\left.\left(\frac{3+i}{4}, \frac{3+i}{4}\right)\,,\, \left(\frac{3+3i}{4}, \frac{3+3i}{4}\right)\right\}.
\]
Blowing up $A$ at these points, we obtain automorphisms $Y \to Y$ and $X \to X$ that we still denote by $\psi$ and $\psi_1$. Then, the pair $(Y, \psi(\mathcal{D}_0))$ has $\psi(\mathcal{D}_0)$ stable under the standard bielliptic involution $\varphi_1$. It follows that the pairs $(Y, \psi(\mathcal{D}_0))$ and $(X, \psi_1(\mathcal{D}))$ then determine ball quotient compactifications biholomorphic to those arising from $(Y, \mathcal{D}_0)$ and $(X, \mathcal{D})$, respectively.

\medskip

\noindent
\textbf{Relationship with Holzapfel's Abelian example.}

\medskip

In \cite[\S 4.6]{Hol04}, Holzapfel constructed an Abelian ball quotient compactification based on the Gaussian integers using the Abelian surface $A$ from above. His construction uses the curves:
\begin{align*}
E_{1}^\prime&=\left(z, z\right) & F_{1}^\prime&=\left(w\,,\, \frac{1}{2}\right) \\
E_{2}^\prime&=\left(-z, z\right) & F_{2}^\prime&=\left(w\,,\, \frac{i}{2}\right) \\
E_{3}^\prime&=\left(iz, z\right) & F_{3}^\prime&=\left(\frac{1}{2}\,,\, z\right) \\
E_{4}^\prime&=(-iz,z) & F_{4}^\prime&=\left(\frac{i}{2}\,,\, z\right)
\end{align*}
These are visibly equivalent to the arrangement $\mathcal{D}_0$ under translation by $\left(\frac{1}{2}, \frac{1}{2}\right)$.

\medskip

\noindent
\textbf{Realizing $\zz/4$.}

\medskip

Let $A^\prime$ be the Abelian surface defined by the subgroup $\zz[1,i] \times \zz[4,i]$ of index $4$ in $\zz[1,i] \times \zz[1, i]$. Consider the lifts to $A^\prime$ of the eight curves on $A$ in \S \ref{sec:MainEx} under the natural map $A^\prime \to A$. This is an arrangement of fourteen elliptic curves in $A^\prime$ given by the equations:
\begin{align*}
E_1^\prime &= (z,z) & F_{2,0}^\prime &= \left(w, \frac{1+i}{2}\right) \\
E_2^\prime &= \left(iz + \frac{1+i}{2},z\right) & F_{2,1}^\prime &= \left(w, \frac{3+i}{2}\right) \\
E_3^\prime &= (-z,z) & F_{2,2}^\prime &= \left(w, \frac{5+i}{2}\right) \\
E_4^\prime &= \left(-iz + \frac{1+i}{2},z\right) & F_{2,3}^\prime &= \left(w, \frac{7+i}{2}\right) \\
F_{1,0}^\prime &= (w, 0) & F_3^\prime &= (0,z) \\
F_{1,1}^\prime &= (w, 1) & F_4^\prime &= \left(\frac{1+i}{2},z\right) \\
F_{1,2}^\prime &= (w, 2) & & \\
F_{1,3}^\prime &= (w, 3) & &
\end{align*}
It is easy to show that these curves meet in $24$ points of $A^\prime$. More precisely, we have intersections as in Table \ref{tb:Z4}.
\begin{table}
\centering
\begin{tabular}{|c|c|}
\hline
Point ($0 \le j \le 3$) & Curves through that point \\
\hline
$(0, j)$ & $E_1^\prime, E_3^\prime, F_{1,j}^\prime, F_3^\prime$ \\
\hline
$\left(\frac{1}{2}, \frac{1}{2} + j\right)$ & $E_1^\prime, E_2^\prime, E_3^\prime, E_4^\prime$ \\
\hline
$\left(\frac{i}{2}, \frac{i}{2} + j\right)$ & $E_1^\prime, E_2^\prime, E_3^\prime, E_4^\prime$ \\
\hline
$\left(\frac{1+i}{2}, \frac{1+i}{2} + j\right)$ & $E_1^\prime, E_3^\prime, F_{2, j}^\prime, F_4^\prime$ \\
\hline
$\left(\frac{1+i}{2}, j\right)$ & $E_2^\prime, E_4^\prime, F_{1,j}^\prime, F_4^\prime$ \\
\hline
$\left(0, \frac{1+i}{2} + j\right)$ & $E_2^\prime, E_4^\prime, F_{2, j}^\prime, F_3^\prime$ \\
\hline
\end{tabular}
\caption{Intersection points on the Abelian surface for the $\zz/4$ bielliptic.}\label{tb:Z4}
\end{table}

If $Y^\prime$ is the blowup of $A^\prime$ at these $24$ points, consider the divisor $\mathcal{D}_0^\prime$ determined by the proper transforms of the above curves. One has that $\overline{c}_1^2(Y^\prime, \mathcal{D}_0^\prime) = 3 \overline{c}_2(Y^\prime, \mathcal{D}_0^\prime)$, which then implies that $Y^\prime \ssm \mathcal{D}_0^\prime$ is a ball quotient with fourteen cusps.

We now consider the map $\varphi : A^\prime \to A^\prime$ of order $4$ given by
\[
\varphi(w, z) = \left(i w + \frac{1+i}{2}, z + 1\right).
\]
One checks directly that we have
\begin{align*}
\varphi(E_j^\prime) &= E_{j+1}^\prime & \varphi(F_3^\prime) &= \varphi(F_4^\prime) \\
\varphi(F_{k,j}^\prime) &= F_{k, j+1}^\prime & \varphi(F_4^\prime) &= \varphi(F_3^\prime)
\end{align*}
with $k = 1,2$ and $j$ considered mod $4$. Let $\pi: A^\prime\rightarrow B^\prime$ be the bielliptic quotient of $A^\prime$ associated with this $\zz / 4$ action, and define
\begin{align*}
G_{1}&=\pi(E^\prime_{j}) & G_{2}&=\pi(F^\prime_{1, j}) \\
G_{3}&=\pi(F^\prime_{2, j}) & G_{4}&=\pi(F^\prime_3)=\pi(F^\prime_4).
\end{align*}
Observe that $G_{1}$ is singular elliptic curve with 4 nodal singularities and 2 singular points of degree four. On the other hand, the elliptic curves $G_2, G_3, G_4$ are smooth. See Figure \ref{fig:Z4} for the arrangement of curves and their intersection points.
\begin{figure}
\hspace*{-2cm}
\begin{tikzpicture}[scale = 1.21]
\useasboundingbox (-6,-3.5) rectangle (6,3.5);
\draw [very thick, blue, rounded corners] (1,1) .. controls (1,7) and (-1,-5) .. (-1,-1);
\draw [very thick, blue, rounded corners] (1,1) .. controls (0.75,5.5) and (-1.25,-6.5) .. (-1,-1);
\draw [line width = 4pt, white, rounded corners] (-1,-1) .. controls (-1,0) and (-1,2) .. (-2.4, 2.4) .. controls (-5,2.2) and (-5,-2.2) .. (-2.4, -2.4) .. controls (-1.9, -2.4) and (-1.6, -2.3) .. (-1, -1) .. controls (-0.5,0.5) and (-0.5,0.5) .. (1,1) .. controls (1.4, 1.1) and (2, 1.5) .. (2.4, 2.4) .. controls (5.2,6) and (5.2,-6) .. (2.4, -2.4) .. controls (1.9, -1.75) and (1.6, -1.5) .. (1,1);
\draw [line width = 4pt, white, rounded corners] (-1,-1) .. controls (-1.6, 1.5) and (-1.9, 1.75) .. (-2.4, 2.4) .. controls (-5.2,6) and (-5.2,-6) .. (-2.4, -2.4) .. controls (-2, -1.5) and (-1.6, -1.1) .. (-1, -1) .. controls (0.5,-0.5) and (0.5,-0.5) .. (1,1) .. controls (1.4, 1.9) and (2, 2.4) .. (2.4, 2.4) .. controls (5,2.2) and (5,-2.2) .. (2.4, -2.4) .. controls (1,-2) and (1,0) .. (1,1);
\draw [very thick, blue, rounded corners] (-1,-1) .. controls (-1,0) and (-1,2) .. (-2.4, 2.4) .. controls (-5,2.2) and (-5,-2.2) .. (-2.4, -2.4) .. controls (-1.9, -2.4) and (-1.6, -2.3) .. (-1, -1) .. controls (-0.5,0.5) and (-0.5,0.5) .. (1,1) .. controls (1.4, 1.1) and (2, 1.5) .. (2.4, 2.4) .. controls (5.2,6) and (5.2,-6) .. (2.4, -2.4) .. controls (1.9, -1.75) and (1.6, -1.5) .. (1,1);
\draw [very thick, blue, rounded corners] (-1,-1) .. controls (-1.6, 1.5) and (-1.9, 1.75) .. (-2.4, 2.4) .. controls (-5.2,6) and (-5.2,-6) .. (-2.4, -2.4) .. controls (-2, -1.5) and (-1.6, -1.1) .. (-1, -1) .. controls (0.5,-0.5) and (0.5,-0.5) .. (1,1) .. controls (1.4, 1.9) and (2, 2.4) .. (2.4, 2.4) .. controls (5,2.2) and (5,-2.2) .. (2.4, -2.4) .. controls (1,-2) and (1,0) .. (1,1);
\draw [very thick, blue] (-1,-1) -- (-1, -1.2);
\draw [very thick, blue] (1,1) -- (1, 1.2);
\draw (0,0) [very thick, red] ellipse (4 and 3);
\draw [very thick, green] (2.4,3.2) -- (2.4,-3.2);
\draw [very thick, yellow!30, orange!70] (-2.4,3.2) -- (-2.4,-3.2);
\filldraw (2.4,2.4) circle [black, radius=0.075cm];
\filldraw (2.4,-2.4) circle [black, radius=0.075cm];
\filldraw (-2.4,2.4) circle [black, radius=0.075cm];
\filldraw (-2.4,-2.4) circle [black, radius=0.075cm];
\filldraw (-1,-1) circle [black, radius=0.075cm];
\filldraw (1,1) circle [black, radius=0.075cm];
\draw (4.8,0) node [blue] {$G_1$};
\draw (-2.4, -3.45) node [yellow!30, orange!70] {$G_2$};
\draw (2.4, -3.45) node [green] {$G_3$};
\draw (0, -3.25) node [red] {$G_4$};
\draw (-3,-2.55) node {$[0,0]$};
\draw (3.3,2.6) node {$\left[\frac{1+i}{2}, \frac{1+i}{2}\right]$};
\draw (3.2,-2.58) node {$\left[0, \frac{1+i}{2}\right]$};
\draw (1.5,0.85) node {$\left[\frac{i}{2}, \frac{i}{2}\right]$};
\draw (-1.5,-0.8) node {$\left[\frac{1}{2}, \frac{1}{2}\right]$};
\draw (-3.15,2.58) node {$\left[\frac{1+i}{2},0\right]$};
\end{tikzpicture}
\caption{The curves $G_i$ on the $\zz / 4$ bielliptic surface $B^\prime$.}\label{fig:Z4}
\end{figure}
Now, let $X$ be the blowup of $B^\prime$ at the $6$ points
\[
\left\{[0, 0]\,,\, \left[\frac{1}{2}, \frac{1}{2}\right]\,,\,\left[\frac{i}{2}, \frac{i}{2}\right]\,,\,\left[\frac{1+i}{2}, 0\right]\,,\, \left[0, \frac{1+i}{2}\right]\,,\, \left[\frac{1+i}{2}, \frac{1+i}{2}\right]\right\}
\]
and let $\{D_{j}\}^{4}_{j=1}$ be the proper transforms of the curves $\{G_{j}\}^{4}_{j=1}$. One easily computes
\[
D^2_{1}=-16 \,,\, D^2_{2}=-2 \,,\, D^2_3=-2 \,,\, D^2_4=-4.
\]
Thus, let $\mathcal{D}$ be the union of these four elliptic curves. Given the pair $(X, \mathcal{D})$, it is immediate to verify that $\overline{c}_1^2(X, \mathcal{D}) = 3\, \overline{c}_2(X, \mathcal{D})$. Thus $X \ssm \mathcal{D}$ is a ball quotient with compactification that is by construction birational to a $\zz / 4$ bielliptic surface.

\medskip

We briefly note that one can relate $\varphi$ to the standard action by translation on the first factor by $\frac{1+i}{2}$. Also, notice that this example is commensurable with the $\zz/2$ example constructed above.

\medskip

\noindent
\textbf{The proof of Theorem \ref{thm:MainExists}.}

\medskip

Above we constructed examples associated with $\zz/2$ and $\zz/4$. It remains to consider $(\zz/2)^2$ and $\zz/4 \times \zz/2$. We claim that any $\zz/4$ bielliptic surface admits a finite \'etale cover by a bielliptic surface of type $\zz/4 \times \zz/2$ or $(\zz/2)^2$. It follows immediately that one can find coverings of the above $\zz/4$ ball quotient whose compactifications are bielliptic surfaces associated with the remaining two groups; see \cite[Lem.\ 1.3]{DS16}. We leave it to the motivated reader to work out these arrangements in coordinates.

To prove the claim, the fundamental group of any $\zz/4$ bielliptic surface is isomorphic to the group of affine transformations of $\cc^2$ generated by:
\begin{align*}
a(w, z) &= (w+1, z) &&& b(w, z) &= (w+i, z) \\
c(w, z) &= (w, z+1) &&& d(w, z) &= (w, z+i) \\
&& e(w,z) &= \left(iw, z + \frac{1}{4}\right) &&
\end{align*}
This has abstract presentation
\begin{align*}
G = \Big\langle a,b,c,d,e\ |\ &[a,b]\,,\,[a,c]\,,\,[a,d]\,,\,[b,c]\,,\,[b,d]\,,\,[c,d] \\
&e a e^{-1} = b\,,\,e b e^{-1} = a^{-1}\,,\,[e,c]\,,\,[e,d]\,,\,e^4=c\Big\rangle.
\end{align*}
One can easily connect this presentation to the $\zz/4$-action on the Abelian surface associated with $\langle a,b,c,d \rangle$ via the exact sequence
\[
1 \to \langle a,b,c,d \rangle \to G \to \zz/4 \to 1
\]
sending $e$ to a generator for $\zz/4$. Note that the given bielliptic surface may not be biholomorphic to the quotient of $\cc^2$ by this particular action, but all $\zz/4$ bielliptic surfaces have isomorphic fundamental groups.

First, consider the subgroup $H$ of $G$ generated by $b a^{-1}$, $a^{-1} b^{-1}$, $c$, $d a^{-1}$, and $e$. This is clearly index two in $G$, and we claim that it determines a $\zz/4 \times \zz/2$ bielliptic. One can compute directly from the action on the abelian surface that a $\zz/4 \times \zz/2$ bielliptic has abstract fundamental group
\begin{align*}
H^\prime = \Big\langle x,y,u,v,r,s\ |\ &[x,y]\,,\,[x,u]\,,\,[x,v]\,,\,[y,u]\,,\,[y,v]\,,\,[u,v] \\
&r x r^{-1} = y\,,\,r y r^{-1} = x^{-1}\,,\,[r,u]\,,\,[r,v]\,,\,r^4=u \\
&[s, x]\,,\,[s,y]\,,\,[s,u]\,,\,[s,v]\,,\,r s r^{-1} = s x^{-1}\,,\,s^2 = x y v\Big\rangle.
\end{align*}
In order to derive this presentation, one can simply look at the $\zz/4\times\zz/2$ bielliptic generated by the commuting automorphisms
\[
r(w, z)=\left(iw \,,\, z+\frac{1}{4}\right), \quad s(w,z)=\left(w+\frac{1+i}{2} \,,\, z+\frac{i}{2}\right)
\]
acting on the abelian surface $\cc^2 / (\zz[i] \times \zz[i])$ with coordinates $(w,z)$. This fits into an exact sequence
\[
1 \to \langle x,y,u,v \rangle \to H^\prime \to \zz/4 \times \zz/2 \to 1,
\]
sending $r$ to a generator for $\zz/4$ and $s$ to a generator for $\zz/2$. One can then identify $H^\prime$ with $H$ by:
\begin{align*}
\langle x,y,u,v \rangle &\to \langle b a^{-1}, a^{-1} b^{-1}, c, d^2 \rangle \\
r &\to e \\
s &\to d a^{-1}
\end{align*}
This gives a $\zz/4 \times \zz/2$ bielliptic as a two-fold \'etale cover of a $\zz/4$ bielliptic.

It is then easy to see that the subgroup of $H$ generated by $b a^{-1}$, $a^{-1} b^{-1}$, $c$, $d a^{-1}$, and $e^2$ has index two in $H$ and determines a two-fold \'etale covering of a $\zz/4 \times \zz/2$ bielliptic by a $(\zz/2)^2$ bielliptic. We leave the details to the reader. This completes the proof of the claim, and hence the theorem.\qed

\section*{Appendix: A ball quotient compactification with $q=0$.}

In this section, using the computer algebra software Magma \cite{Magma}, we indicate that there exists a smooth toroidal compactification of a ball quotient with irregularity $q = 0$.\footnote{Code is available on the authors' webpages.} This shows that the previous belief that such an example cannot exist is incorrect, motivating further study of ball quotient compactifications of irregularity zero. In particular, it remains an interesting question whether a rational, K3, or Enriques surface could possibly be a ball quotient compactification.

Our group is a subgroup of the Picard modular group associated with the ring $\zz[i]$ of Gaussian integers. A presentation for this lattice, determined by Falbel, Francsics, and Parker \cite{FFP} is
\[
\Gam = \left\langle i \,,\, q \,,\, t\ |\ i^2 \,,\, q^2 \,,\, (i q)^3 \,,\, (i t)^{12} \,,\, (i q t)^8 \,,\, [(i t)^3 , t] \,,\, [q , t] \right\rangle.
\]
The associated ball quotient orbifold has Euler number $1/32$ \cite[\S 8.3]{FFP}. Consider finite group
\begin{align*}
F = \Big\langle\alpha \,,\, \beta \,,\, \gamma\ |\ &\alpha^2 \,,\, \beta^2 \,,\, \gamma^4 \,,\, [\beta,\gamma] \,,\, (\alpha \beta)^3 \,,\, [\alpha,\gamma]^3 \,,\, (\alpha \beta \gamma)^8\,, \\
&\gamma^{-2} \beta \alpha \gamma^{-1} \beta \gamma^{-1} \alpha \gamma^{-2} \beta \alpha \,,\, \alpha \gamma \beta \gamma^2 \beta \gamma \alpha \gamma^4 \,,\, (\alpha \gamma^{-1})^3 (\alpha \gamma)^3 \Big\rangle
\end{align*}
of order $1536$. Up to automorphisms of $F$ there is a unique homomorphism
\[
\rho : \Gamma \to F.
\]
Let $\Delta$ be the kernel of $\rho$. Magma shows $\Delta$ has finite abelianization $(\zz/2)^9 \times (\zz/4)^6$. A presentation for $\Delta$ has $15$ generators and $149$ relations, which we omit.

Using \cite[\S 8.3]{FFP}, one can explicitly determine the conjugacy classes of finite-order elements in $\Gam$. Since $\Del$ is normal in $\Gam$, one can then check directly with Magma that $\Del$ is torsion-free by verifying that $\Del$ contains none of these representatives. Therefore, $\Del$ is the fundamental group of a smooth noncompact ball quotient of Euler number $1536/32 = 48$.

To show that $\B^2 / \Del$ admits a smooth toroidal compactification, it suffices to show that all parabolic elements in $\Del$ are rotation-free. The orbifold $\B^2 / \Gam$ has a unique cusp, hence a unique conjugacy class of parabolic subgroups. A representative $B$ for this group is described in \cite[\S 8.2]{FFP}, which includes both an abstract presentation for $B$ and its generators as words in the generators for $\Gam$.

Consider the quotient homomorphism
\[
\rho : \Gam \to F = \Gam / \Del.
\]
Since $\B^2 / \Delta$ is a Galois orbifold cover of $\B^2 / \Gam$ with group $F$, we see that $\B^2 / \Gam$ has $[F : \rho(B)] = 24$ cusps, which all have fundamental group isomorphic to the subgroup $\ker(\rho|_B)$ of $B$. Using the above abstract presentation for $B$ and knowing $\rho(r)$, $\rho(q)$, and $\rho(t)$, Magma can compute that the kernel has presentation
\[
\langle \delta_1 \,,\, \delta_2 \,,\, \delta_3\ |\ [\delta_1, \delta_2] \delta_3^{-4} \rangle.
\]
As in \cite[Prop.\ 4.2.12]{HolzBook}, this shows that each cusp admits a smooth compactification by a smooth elliptic curve of self-intersection $-4$.

Let $X$ be the smooth toroidal compactification of $\B^2 / \Del$. Since $\Del$ has finite abelianization, it follows that $q(X) = 0$ (e.g., see \cite{DS17}). Let $D$ be the compactifying divisor, so $X \ssm D = \B^2 / \Del$. Then, from $\overline{c}_1^2 = 3 \overline{c}_2$, we see that
\[
K_X^2 - D^2 = 3\, c_2(X) = 144.
\]
Since $D$ consists of $24$ disjoint elliptic curves of self-intersection $-4$, we obtain $K_X^2 = 48$. It is then easy to complete the following table of characteristic numbers:
\begin{align*}
c_1^2(X) = c_2(X) &= 48 \\
\chi(\mathcal{O}_X) &= 8 \\
p_g(X) &= 7
\end{align*}
It is not hard to see from here that $X$ has minimal model of general type, but we do not know if $X$ is minimal. It would be interesting to understand this surface better.



\begin{thebibliography}{ELMNPM}

\bibitem[Ada05]{Ada85} C. Adams, Thrice-punctured spheres in hyperbolic $3$ manifolds. \textit{Trans. Amer. Math. Soc} \textbf{287} (1985), no. 2, 645-656.

\bibitem[AMRT10]{Ash} A. Ash, D. Mumford, M. Rapoport, Y.-S. Tai,
Smooth compactifications of locally symmetric varieties. Second
edition. Cambridge Mathematical Library. \textit{Cambridge
University Press, Cambridge}, 2010.

\bibitem[BHH87]{HH} G. Barthel, F. Hirzebruch, T. H\"ofer, Geradenkonfigurationen und Algebraische Fl\"achen, Aspects of Mathematics, D4. \textit{Friedr. Vieweg \& Sohn, Braunschweig}, 1987.

\bibitem[Bea96]{Bea} A. Beauville, Complex Algebraic Surfaces, Second Edition,
London Mathematical Society Students Texts, 34. \textit{Cambridge University Press, Cambridge}, 1996.

\bibitem[Ber00]{Ber00} N. Bergeron, Premier nombre de Betti et spectre du laplacian de certaines vari\'et\'es hyperboliques. \textit{L'Ensegnement Math.} \textbf{46} (2000), 109-137.

\bibitem[DD15]{DiC} G. Di Cerbo, L. F. Di Cerbo, Effective results for complex hyperbolic manifolds.
\textit{J. Lond. Math. Soc. (2)} \textbf{91} (2015), 89-104.

\bibitem[DS16]{DS16} L. F. Di Cerbo, M. Stover, Multiple realizations of varieties as ball quotient compactifications.
\textit{Michigan Math. J.} \textbf{65} (2016), no. 2, 441--447.

\bibitem[DS17]{DS17} L. F. Di Cerbo, M. Stover, Bielliptic ball quotient compactifications and lattices in $\mathrm{PU}(2,1)$ with finitely generated commutator subgroup. \textit{Ann. Inst. Fourier} \textbf{67} (2017), no. 1, 315--328.

\bibitem[DS18]{DS15} L. F. Di Cerbo, M. Stover, Classification and arithmeticity of toroidal compactifications with $3\overline{c}_{2}=\overline{c}^{2}_{1}=3$. \textit{Geom. Topol.} \textbf{22} (2018), no. 4, 2465--2510.

\bibitem[FFP]{FFP} E. Falbel, G. Francsics, J. R. Parker, The geometry of the Gauss--Picard modular group.
\textit{Math. Ann.} \textbf{349} (2011), no. 2, 459--508.

\bibitem[Hir84]{Hir84} F. Hirzebruch, Chern numbers of algebraic surfaces: an example. \textit{Math. Ann.} \textbf{266} 
(1984), 351-356.

\bibitem[Hol86]{Holzapfel} R. P. Holzapfel, Chern numbers of algebraic surfaces--Hirzebruch's examples are Picard modular surfaces.
\textit{Math. Nachr.}, \textbf{126} (1986), 255--273.

\bibitem[Hol98]{HolzBook} R. P. Holzapfel, Ball and surface arithmetics. Aspects of Mathematics, \textbf{29}. Friedr. Vieweg \& Sohn, 1998.

\bibitem[Hol04]{Hol04} R.-P. Holzapfel, Complex hyperbolic surfaces of abelian type. \textit{Serdica Math. J.} \textbf{30} 
(2004), no. 2-3, 207-238.

\bibitem[KN96]{Kobayashi} S. Kobayashi, K. Nomizu, Foundations of differential geometry. Vol. II. Reprint of the 1969 original. Wiley Classic Library. A Wiley-Interscience Publication. \textit{John Wiley \& Sons, Inc., New York}, 1996.

\bibitem[Mag]{Magma} W. Bosma, J. Cannon, C. Playoust, The {M}agma algebra system. {I}. {T}he user language. \textit{J. Symbolic Comput.}, Computational algebra and number theory (London, 1993), \textbf{24} (1997), no. 3-4, 235--265.

\bibitem[Mat02]{Matsuki} K. Matsuki, Introduction to the Mori program. Universitext. \textit{Springer-Verlag, New York,} 2002.

\bibitem[Mok12]{Mok} N. Mok, Projective algebraicity of minimal compactifications of complex-hyperbolic space forms of finite volume. Perspective in analysis, geometry, and topology, 331-354,
\emph{Prog. Math.,} \textbf{296}, \emph{Birkh\"auser/Springer, New
York}, 2012.

\bibitem[MVZ09]{Muller} S. M\"uller-Stach, E. Viehweg, K. Zuo, Relative proportionality for subvarieties of moduli spaces of $K3$ and abelian surfaces. \textit{Pure Appl. Math. Q.} \textbf{5} (2009), no. 3, 1161-1199.

\bibitem[Mum77]{Mumford} D. Mumford, Hirzebruch's Proportionality Theorem in the Non-Compact Case.
\textit{Invent. Math.}, \textbf{42} (1977), 239--272.

\bibitem[Sco78]{Scott} P. Scott, Subgroups of surface groups are almost geometric.
\textit{J. London Math. Soc.} (2) \textbf{17} (1978), no. 3, 555--565.

\end{thebibliography}
\end{document}